\documentclass[11pt, a4paper]{amsart}

\usepackage{amsmath,amssymb,hyperref}
\usepackage{amsfonts}
\usepackage{enumerate}
\usepackage{amsmath, amsthm}
\usepackage{amscd}
\input xy
\xyoption{all}

      \theoremstyle{plain}
      \newtheorem{theorem}{Theorem}[section]
      \newtheorem{lemma}[theorem]{Lemma}
      
      \newtheorem{proposition}[theorem]{Proposition}
      \theoremstyle{definition}
      \newtheorem{definition}[theorem]{Definition}

      \theoremstyle{remark}
      \newtheorem{remark}[theorem]{Remark}

\newcommand{\vol}{\mathrm{Vol}}
\newcommand{\X}{\mathcal{X}}

      \makeatletter
      \def\@setcopyright{}
      \def\serieslogo@{}
      \makeatother

\begin{document}

%



   \author{Sungwoon Kim}
   \address{School of Mathematics,
   KIAS, Hoegiro 85, Dongdaemun-gu,
   Seoul, 130-722, Republic of Korea}
   \email{sungwoon@kias.re.kr}






   \title[Volume of representations]{On the equivalence of the definitions of volume of representations}


\begin{abstract}
Let $G$ be a rank $1$ simple Lie group and $M$ be a connected, orientable, aspherical, tame manifold. Assume that each end of $M$ has amenable fundamental group. There are several definitions of volume of representations of $\pi_1(M)$ into $G$. We give a new definition of volume of representations and furthermore, show that all definitions so far are equivalent. 
\end{abstract}

\footnotetext[1]{2000 {\sl{Mathematics Subject Classification.}}
53C35, 53C24}

\footnotetext[2]{{\sl{Key words and phrases.}}
semisimple Lie group, volume of representations}

\footnotetext[3]{This research was supported by the Basic Science Research Program through the National Research Foundation of Korea (NRF) funded by the Ministry of Education, Science and Technology (NRF-2012R1A1A2040663)}


   \keywords{}

   \thanks{}
   \thanks{}

   \dedicatory{}

   \date{}


   \maketitle



\section{Introduction}

Let $G$ be a semisimple Lie group and $\X$ the associated symmetric space of dimension $n$. 
Let $M$ be a connected, orientable, aspherical, tame manifold of the same dimension as $\X$. First assume that $M$ is compact. To each representation $\rho :\pi_1(M) \rightarrow G$, one can associate a volume of $\rho$ in the following way. First, associate a flat bundle $E_\rho$ over $M$ with fiber $\X$ to $\rho$. Since $\X$ is contractible, there always exists a section $s : M \rightarrow E_\rho$. Let $\omega_{\X}$ be the Riemannian volume form on $\X$. One may think of $\omega_{\X}$ as a closed differential form on $E_\rho$ by spreading $\omega_{\X}$ over the fibers of $E_\rho$. Then the volume of $\rho$ is defined by 
$$\vol(\rho)=\int_M s^*\omega_{\X}.$$
Since any two sections are homotopic to each other, the volume $\vol(\rho)$ does not depend on the choice of section. 

The volume of representations has been used to characterize discrete faithful representations. 
Let $\Gamma$ be a uniform lattice in $G$. Then the volume of representations satisfies a Milnor-Wood type inequality. More precisely, it holds that for any representation $\rho :\Gamma\rightarrow G$, \begin{eqnarray}\label{MWinequality} |\vol(\rho)| \leq \vol(\Gamma\backslash \X).\end{eqnarray}
Furthermore, equality holds in (\ref{MWinequality}) if and only if $\rho$ is discrete and faithful. This is the so-called \emph{volume rigidity theorem}. Goldman \cite{Go92} proved the volume rigidity theorem in higher rank case and, Besson, Courtois and Gallot \cite{BCG07} proved the theorem in rank $1$ case.

Now assume that $M$ is noncompact. 
Then the definition of volume of representations as above is not valid anymore since some problems of integrability arise. So far, three definitions of volume of representations have been given under some conditions on $M$. Let us first fix the following notations throughout the paper.

\smallskip
\noindent {\bf Setup.} Let $M$ be a noncompact, connected, orientable, aspherical, tame manifold. Denote by $\overline M$ the compact manifold with boundary whose interior is homeomorphic to $M$. Assume that each connected component of $\partial \overline M$ has amenable fundamental group. Let $G$ be a rank $1$ semisimple Lie group with trivial center and no compact factors. Let $\X$ be the associated symmetric space of dimension $n$. Assume that $M$ has the same dimension as $\X$.

\smallskip

First of all, Dunfield \cite{Du99} introduced the notion of pseudo-developing map to define the volume of representations of a nonuniform lattice $\Gamma$ in $\mathrm{SO}(3,1)$. It was successful to make an invariant associated with a representation $\rho :\Gamma \rightarrow \mathrm{SO}(3,1)$ but he did not prove that the volume of representations does not depend on the chosen pseudo-developing map.
After that, Francaviglia \cite{Fr04} proved the well-definedness of the volume of representations. Then Francaviglia and Klaff \cite{FK06} extended the definition of volume of representations and the volume rigidity theorem to general nonuniform hyperbolic lattices. 
We call the definition of volume of representations via pseudo-developing map {\bf D1}. For more detail about {\bf D1}, see \cite{FK06} or Section \ref{sec:pseudo}.

The second definition {\bf D2} of volume of representations was given by Bucher, Burger and Iozzi \cite{BBI}, which generalizes the one introduced in \cite{BIW10} for noncompact surfaces. They used the bounded cohomology theory to make an invariant associated with a representation. 
Given a representation $\rho : \pi_1(M) \rightarrow G$, one can not get any information from the pull-back map in degree $n$ in continuous cohomology, $\rho^*_c : H^n_c(G,\mathbb R) \rightarrow H^n(\pi_1(M),\mathbb R)$, since $H^n(\pi_1(M),\mathbb R) \cong H^n(M,\mathbb R)$ is trivial. However the situation is different in continuous bounded cohomology. Not only may be the pull-back map $\rho^*_b : H^n_{c,b}(G,\mathbb R) \rightarrow H^n_b(\pi_1(M),\mathbb R)$ nontrivial but also encodes subtle algebraic and topological properties of a representation such as injectivity and discreteness. 
Bucher, Burger and Iozzi \cite{BBI} gave a proof of the volume rigidity theorem for representations of hyperbolic lattices from the point of view of bounded cohomology. We refer the reader to \cite{BBI} or Section \ref{sec:bounded} for further discussion about {\bf D2}.


Recently, S. Kim and I. Kim \cite{KK14} give a new definition, called {\bf D3}, of volume of representations in the case that $M$ is a complete Riemannian manifold with finite Lipschitz simplicial volume. See \cite{KK14} or Section \ref{sec:lipschitz} for the exact definition of {\bf D3}. In {\bf D3}, it is not necessary that each connected component of $\partial \overline M$ has amenable fundamental group while the amenable condition on $\partial \overline M$ is necessary in {\bf D2}. They only use the bounded cohomology and $\ell^1$-homology of $M$. It is quite useful to define the volume of representations in the case that the amenable condition on $\partial \overline M$ does not hold. They give a proof of the volume rigidity theorem for representations of lattices in an arbitrary semisimple Lie group in their setting.

In this note, we will give another definition of volume of representations, called {\bf D4}. In {\bf D4}, $\rho$-equivariant maps are involved as {\bf D1} and the bounded cohomology of $M$ is involved as {\bf D2} and {\bf D3}. In fact, {\bf D4} seems a kind of definition connecting the other definitions {\bf D1}, {\bf D2} and {\bf D3}. Eventually we show that all definitions are equivalent.

\begin{theorem}\label{thm:main}
Let $G$ be a rank $1$ simple Lie group with trivial center and no compact factors. Let $M$ be a noncompact, connected, orientable, aspherical, tame manifold. Suppose that each end of $M$ has amenable fundamental group. Then all definitions {\bf D1}, {\bf D2} and {\bf D3} of volume of representations of $\pi_1(M)$ into $G$ are equivalent. Furthermore if $M$ admits a complete Riemannian metric with finite Lipschitz simplicial volume, all definitions {\bf D1}, {\bf D2}, {\bf D3} and {\bf D4} are equivalent.
\end{theorem}

The paper is organized as follows: For our proof, we recall the definitions of volume of representations in the order {\bf D2, D4, D1, D3}.
In Section \ref{sec:bounded}, we first recall the definition {\bf D2}. In Section \ref{sec:cone}, we give the definition {\bf D4} and then prove that {\bf D2} and {\bf D4} are equivalent. In Section \ref{sec:pseudo}, after recalling the definition {\bf D1}, we show the equivalence of {\bf D1} and {\bf D4}. Finally in Section \ref{sec:lipschitz}, we complete the proof of Theorem \ref{thm:main} by proving that {\bf D3} and {\bf D4} are equivalent.

\section{Bounded cohomology and Definition {\bf D2}}\label{sec:bounded}


We choose the appropriate complexes for the continuous cohomology and continuous bounded cohomology of $G$ for our purpose.
Consider the complex $C^*_c(\mathcal X,\mathbb R)_\mathrm{alt}$ with the homogeneous coboundary operator, where  
$$C^k_c(\mathcal X,\mathbb R)_\mathrm{alt} =\{ f : \mathcal X^{k+1} \rightarrow \mathbb R \ | \ f \text{ is continuous and alternating} \}.$$
The action of $G$ on $C^k_c(\mathcal X,\mathbb R)_\mathrm{alt}$ is given by $$g\cdot f (x_0,\ldots,x_k)=f(g^{-1}x_0,\ldots,g^{-1}x_k).$$ Then the continuous cohomology $H^*_c(G,\mathbb R)$ can be isomorphically computed by the cohomology of the $G$-invariant complex $C^*_c(\mathcal X,\mathbb R)_\mathrm{alt}^G$ (see \cite[Chapitre III]{Gu80}). According to the Van Est isomorphism \cite[Proposition IX.5.5]{BW}, the continuous cohomology $H^*_c(G,\mathbb R)$ is isomorphic to the set of $G$-invariant differential forms on $\X$. Hence in degree $n$, $H^n_c(G,\mathbb R)$ is generated by the Riemannian volume form $\omega_{\mathcal X}$ on $\X$. 

Let $C^k_{c,b}(\X,\mathbb R)_\mathrm{alt}$ be subcomplex of continuous, alternating, bounded real valued functions on $\X^{k+1}$.
The continuous bounded cohomology $H^*_{c,b}(G,\mathbb R)$ is obtained by the cohomology of the $G$-invariant complex $C^*_{c,b}(\X,\mathbb R)_\mathrm{alt}^G$ (see \cite[Corollary 7.4.10]{Mo01}). The inclusion of complexes $C^*_{c,b}(\X,\mathbb R)^G_\mathrm{alt} \subset C^*_{c}(\X,\mathbb R)^G_\mathrm{alt}$ induces a comparison map $H^*_{c,b}(G,\mathbb R) \rightarrow H^*_{c}(G,\mathbb R)$. 

Let $Y$ be a countable CW-complex. Denote by $C^k_b(Y,\mathbb R)$ the complex of bounded real valued $k$-cochains on $Y$. 
For a subspace $B \subset Y$, let $C^k_b(Y,B,\mathbb R)$ be the subcomplex of those bounded $k$-cochains on $Y$ that vanish on simplices with image contained in $B$. The complexes $C^*_b(Y,\mathbb R)$ and $C^*_b(Y,B,\mathbb R)$ define the bounded cohomologies $H^*_b(Y,\mathbb R)$ and $H^*_b(Y,B,\mathbb R)$ respectively. For our convenience, we give another complex which computes the bounded cohomology $H^*_b(Y,\mathbb R)$ of $Y$. Let $C^k_b(\widetilde Y,\mathbb R)_\mathrm{alt}$ denote the complex of bounded, alternating real valued Borel functions on $(\widetilde Y)^{k+1}$. The $\pi_1(Y)$-action on $C^*_b(\widetilde Y,\mathbb R)_\mathrm{alt}$ is defined as the $G$-action on $C^*_c(\mathcal X,\mathbb R)$. Then Ivanov \cite{Iva85} proved that the $\pi_1(Y)$-invariant complex $C^*_b(\widetilde Y,\mathbb R)_\mathrm{alt}^{\pi_1(Y)}$ defines the bounded cohomology of $Y$.

Bucher, Burger and Iozzi \cite{BBI} used bounded cohomology to define the volume of representations.
Let $\overline M$ be a connected, orientable, compact manifold with boundary. Suppose that each component of $\partial \overline M$ has amenable fundamental group. In that case, it is proved in \cite{BBIPP,KK} that the natural inclusion $i:(\overline M,\emptyset) \rightarrow (\overline M,\partial \overline M)$ induces an isometric isomorphism in bounded cohomology, $$i_b^* : H^*_b(\overline M, \partial \overline M,\mathbb R) \rightarrow H^*_b(\overline M,\mathbb R),$$ in degrees $* \geq 2$. Noting the remarkable result of Gromov \cite[Section 3.1]{Gro82} that the natural map  $H^n_b(\pi_1(\overline M),\mathbb R)\rightarrow H^n_b(\overline M,\mathbb R)$ is an isometric isomorphism in bounded cohomology, for a given representation $\rho : \pi_1(M) \rightarrow G$ we have a map $$\rho^*_b : H^n_{c,b}(G,\mathbb R) \rightarrow H^n_b(\pi_1(\overline M),\mathbb R) \cong H^n_b(\overline M,\mathbb R) \cong H^n_b(\overline M,\partial \overline M,\mathbb R).$$

The $G$-invariant Riemannian volume form $\omega_\X$ on $\X$ gives rise to a continuous bounded cocycle $\Theta :\X^{n+1} \rightarrow \mathbb R$ defined by $$\Theta(x_0,\ldots,x_n)=\int_{[x_0,\ldots,x_n]}\omega_\X,$$ where $[x_0,\ldots,x_n]$ is the geodesic simplex with ordered vertices $x_0,\ldots,x_n$ in $\X$. The boundedness of $\Theta$ is due to the fact that the volume of geodesic simplices in $\X$ is uniformly bounded from above \cite{IY82}. Hence the cocycle $\Theta$ induces a continuous cohomology class $[\Theta]_c \in H^n_c(G,\mathbb R)$ and moreover, a continuous bounded cohomology class $[\Theta]_{c,b} \in H^n_{c,b}(G,\mathbb R)$.
The image of $((i^*_b)^{-1} \circ \rho^*_b)[\Theta]_{c,b}$ via the comparison map $c : H^n_b(\overline M,\partial \overline M,\mathbb R) \rightarrow H^n(\overline M,\partial \overline M,\mathbb R)$ is an ordinary relative cohomology class. Its evaluation on the relative fundamental class $[\overline M,\partial \overline M]$ gives an invariant associated with $\rho$.

\begin{definition}[{\bf D2}]
For a representation $\rho : \pi_1(M) \rightarrow G$, define an invariant $\vol_2(\rho)$ by
$$\mathrm{Vol}_2(\rho)= \left\langle (c\circ (i^*_b)^{-1} \circ \rho^*_b) [\Theta]_{c,b}, [\overline M, \partial \overline M] \right\rangle.$$
\end{definition}

In the definition {\bf D2}, a specific continuous bounded volume class $[\Theta]_{c,b}$ in $H^n_{c,b}(G,\mathbb R)$ is involved. The question is naturally raised as to whether, if another continuous bounded volume class is used in {\bf D2} instead of $[\Theta]_{c,b}$, the value of the volume of representations changes or not. One could expect that the definition {\bf D2} does not depend on the choice of continuous bounded volume class but it seems not easy to get an answer directly. It turns out that {\bf D2} is independent of the choice of continuous bounded volume class. For a proof, see Section \ref{sec:lipschitz}.

\begin{proposition}\label{prop:indepwb}
The definition {\bf D2} does not depend on the choice of continuous bounded volume class, that is, for any two continuous bounded volume classes $\omega_b$, $\omega_b' \in H^n_{c,b}(G,\mathbb R)$,
$$\left\langle (c\circ (i^*_b)^{-1} \circ \rho^*_b) (\omega_b), [\overline M, \partial \overline M] \right\rangle=\left\langle (c\circ (i^*_b)^{-1} \circ \rho^*_b) (\omega_b'), [\overline M, \partial \overline M] \right\rangle.$$
\end{proposition}

Bucher, Burger and Iozzi proved the volume rigidity theorem for hyperbolic lattices as follows.

\begin{theorem}[Bucher, Burger and Iozzi, \cite{BBI}]
Let $n\geq 3$. Let $i :\Gamma \hookrightarrow \mathrm{Isom}^+(\mathbb H^n)$ be a lattice embedding and let $\rho:\Gamma \rightarrow \mathrm{Isom}^+(\mathbb H^n)$ be any representation. Then $$| \vol_2(\rho)| \leq |\vol_2(i)|=\vol(\Gamma \backslash \mathbb H^n),$$
with equality if and only if $\rho$ is conjugated to $i$ by an isometry.
\end{theorem}

\section{New definition {\bf D4}}\label{sec:cone}

In this section we give a new definition of volume of representations.
It will turn out that the new definition is useful in proving that all definitions of volume of representations are equivalent.

\subsection{End compactification}
Let $\widehat{M}$ be the end compactification of $M$ obtained by adding one point for each end of $M$. 
Let $\widetilde M$ denote the universal cover of $M$. Let $\widehat{\widetilde{M}}$ denote the space obtained by adding to $\widetilde{M}$ one point for each lift of each end of $M$. The points added to $M$(resp. $\widetilde M$) are called \emph{ideal points} of $M$(resp. $\widetilde M$). Denote by $\partial \widehat M$(resp. $\partial \widehat{\widetilde M}$) the set of ideal points of $M$(resp. $\widetilde M$). Let $p : \widetilde{M} \rightarrow M$ be the universal covering map. The covering map $p : \widetilde{M} \rightarrow M$ extends to a map $\widehat{p}: \widehat{\widetilde{M}} \rightarrow \widehat{M}$ and moreover, the action of $\pi_1(M)$ on $\widetilde{M}$ by covering transformations induces an action on $\widehat{\widetilde{M}}$. The action on $\widehat{\widetilde{M}}$ is not free because each point of $\partial \widehat{\widetilde{M}}$ is stabilized by some peripheral subgroup of $\pi_1(M)$.

Note that $\widehat M$ can be obtained by collapsing each connected component of $\partial \overline M$ to a point.
Similarly, $\widehat{\widetilde M}$ can be obtained by collapsing each connected component of $\bar p^{-1}(\partial \overline M)$ to a point where $\bar p : \widetilde{\overline M} \rightarrow \overline M$ is the universal covering map.
We denote the collapsing map by $\pi : \widetilde {\overline M} \rightarrow \widehat{\widetilde M}$.

One advantage of $\widehat M$ is the existence of a fundamental class in singular homology.
While the top dimensional singular homology of $M$ vanishes, the top dimensional singular homology of $\widehat M$ with coefficients in $\mathbb Z$ is isomorphic to $\mathbb Z$. Moreover, it can be easily seen that $H_*(\widehat M,\mathbb R)$ is isomorphic to $H_*(\overline M,\partial \overline M,\mathbb R)$ in degree $* \geq 2$. Hence the fundamental class of $\widehat M$ is well-defined and denote it by $[\widehat M]$.

\subsection{The cohomology groups}
Let $Y$ be a topological space and suppose that a group $L$ acts continuously on $Y$. 
Then the cohomology group $H^*(Y;L,\mathbb R)$ associated with $Y$ and $L$ is defined in the following way. Our main reference for this cohomology is \cite{Du}.

For $k>0$, define $$F^k_\mathrm{alt}(Y,\mathbb R)=\{ f : Y^{k+1} \rightarrow \mathbb R \ | \ f \text{ is alternating} \}.$$ Let $F^k_\mathrm{alt}(Y,\mathbb R)^L$ denote the subspace of $L$-invariant functions where the action of $L$ on $F^k_\mathrm{alt}(Y,\mathbb R)$ is given by 
$$(g \cdot f)(y_0,\ldots,y_k)=f(g^{-1}y_0,\ldots, g^{-1}y_k),$$ for $f \in F^k_\mathrm{alt}(Y)$ and $g \in L$.
Define a coboundary operator $\delta_k : F^k_\mathrm{alt}(Y,\mathbb R) \rightarrow F^{k+1}_\mathrm{alt}(Y,\mathbb R)$ by the usual formula
$$(\delta_k f)(y_0,\ldots,y_{k+1})=\sum_{i=0}^{k+1} (-1)^i f(y_0,\ldots, \hat y_i,\ldots, y_{k+1}).$$
The coboundary operator restricts to the complex $F^*_\mathrm{alt}(Y,\mathbb R)^L$. The cohomology $H^*(Y;L,\mathbb R)$ is defined as the cohomology of this complex. Define $F^*_{\mathrm{alt},b}(Y,\mathbb R)$ as the subspace of $F^*_\mathrm{alt}(Y,\mathbb R)$ consisting of bounded alternating functions. Clearly the coboundary operator restricts to the complex $F^*_{\mathrm{alt},b}(Y,\mathbb R)^L$ and so it defines a cohomology, denoted by $H^*_b(Y;L,\mathbb R)$. 
In particular, for a manifold $M$, the cohomology $H^*(\widetilde M;\pi_1(M),\mathbb R)$ is actually isomorphic to the group cohomology $H^*(\pi_1(M),\mathbb R)$ and, $H^*_b(\widetilde M;\pi_1(M), \mathbb R)$ is isomorphic to the bounded cohomology $H^*_b(\pi_1(M),\mathbb R)$.

\begin{remark}\label{remark1}
Let $L$ and $L'$ be groups acting continuously on topological spaces $Y$ and $Y'$, respectively. Given a homomorphism $\rho :L \rightarrow L'$, any $\rho$-equivariant continuous map $P:Y \rightarrow Y'$ defines a chain map $$P^* :F^*_\mathrm{alt}(Y',\mathbb R)^{L'}\rightarrow F^*_\mathrm{alt}(Y,\mathbb R)^L.$$
Thus it gives a morphism in cohomology. Let $Q:Y \rightarrow Y'$ be another $\rho$-equivariant map. For each $k>0$, one may define $$H_k (y_0,\ldots,y_k)=\sum_{i=0}^i (-1)^k (P(y_0),\ldots,P(y_i),Q(y_i),\ldots, Q(y_k)).$$
Then by a straightforward computation, $$(\partial_{k+1}H_k + H_{k-1}\partial_k)(y_0,\ldots,y_k)=(P(y_0),\ldots,P(y_k))-(Q(y_0),\ldots, Q(y_k)).$$
It follows from the above identity that for any cocycle $f \in F^k_\mathrm{alt}(Y',\mathbb R)^{L'}$, $$ (P^*f - Q^*f)(y_0,\ldots,y_k)=\delta_k (f \circ H_{k-1})(y_0,\ldots,y_k).$$
From this usual process in cohomology theory, one could expect that $P$ and $Q$ induce the same morphism in cohomology.
However, since $f \circ H_{k-1}$ may be not alternating, $P$ and $Q$ may not induce the same morphism in cohomology.
\end{remark}

Recalling that $\Theta :\X^{n+1} \rightarrow \mathbb R$ is a $G$-invariant continuous bounded alternating cocycle, it yields a bounded cohomology class $[\Theta]_b \in H^n_b(\X;G,\mathbb R).$
Let $\overline \X$ be the compactification of $\X$ obtained by adding the ideal boundary $\partial \X$.
Extending the $G$-action on $\X$ to $\overline{\X}$, we can define a cohomology $H^*(\overline{\X};G,\mathbb R)$ and bounded cohomology $H^*_b(\overline{\X};G,\mathbb R)$. In rank $1$ case, since the geodesic simplex is well-defined for any $(n+1)$-tuple of points of $\overline X$, the cocycle $\Theta$ can be extended to a $G$-invariant alternating bounded cocycle $\overline \Theta :\overline X^{n+1} \rightarrow \mathbb R$. Hence $\overline \Theta$ determines a cohomology class $[\overline \Theta] \in H^n(\overline \X;G,\mathbb R)$ and $[\overline \Theta]_b \in H^n_b(\overline \X;G,\mathbb R)$.

Let $\widehat D:\widehat{\widetilde{M}} \rightarrow \overline \X$ be a $\rho$-equivariant continuous map whose restriction to $\widetilde{M}$ is a $\rho$-equivariant continuous map from $\widetilde M$ to $\X$. We will consider only such kinds of equivariant maps throughout the paper. Denote by $D : \widetilde M \rightarrow \mathcal X$ the restriction of $\widehat D$ to $\widetilde M$.
Then $\widehat D$ induces a homomorphism in cohomology, $$\widehat D^* : H^n(\overline \X;G,\mathbb R) \rightarrow H^n(\widehat{\widetilde{M}};\pi_1(M),\mathbb R).$$ 
Note that the action of $\pi_1(M)$ on $\widehat{\widetilde M}$ is not free and hence $H^*(\widehat{\widetilde{M}};\pi_1(M),\mathbb R)$ may not be isomorphic to $H^*(\widehat M,\mathbb R)$. 
Let $H^*_{simp}(\widehat M,\mathbb R)$ be the simplicial cohomology induced from a simplicial structure on $\widehat M$.
Then there is a natural restriction map $H^*(\widehat{\widetilde M};\pi_1(M),\mathbb R) \rightarrow H^*_{simp}(\widehat M,\mathbb R)  \cong H^*(\widehat M,\mathbb R)$.
Thus we regard the cohomology class $\widehat D^*[\overline \Theta]$ as a cohomology class of $H^n(\widehat M,\mathbb R)$.
Let $[\widehat M]$ be the fundamental cycle in $H_n(\widehat M,\mathbb R)\cong \mathbb R$.

\begin{definition}[{\bf D4}] Let $D:\widetilde M \rightarrow \X$ be a $\rho$-equivariant continuous map which is extended to a $\rho$-equivariant map $\widehat D : \widehat{\widetilde{M}} \rightarrow \overline \X$. Then we define an invariant $\vol_4(\rho,D)$ by
$$\vol_4(\rho,D)=\langle \widehat D^*[\overline{\Theta}], [\widehat M] \rangle.$$
\end{definition}

As observed before, $\widehat D^*[\overline{\Theta}]$ may depend on the choice of $\rho$-equivariant map. However it turns out that the value $\vol_4(\rho,D)$ is independent of the choice of $\rho$-equivariant continuous map as follows.

\begin{proposition}\label{prop:1}
Let $\rho :\pi_1(M) \rightarrow G$ be a representation. Then
$$\vol_2(\rho)=\vol_4(\rho,D).$$
\end{proposition}

\begin{proof}
Reminding that the continuous bounded cohomology $H^*_{c,b}(G,\mathbb R)$ can be computed isomorphically from the complex $C^*_{c,b}(\mathcal X,\mathbb R)_\mathrm{alt}$, there is the natural inclusion $C^*_{c,b}(\mathcal X,\mathbb R)_\mathrm{alt} \subset F^*_{\mathrm{alt},b}(\mathcal X,\mathbb R)$. Denote the homomorphism in cohomology induced from the inclusion by $i_G : H^k_{c,b}(G,\mathbb R)\rightarrow H^k_b(\X;G,\mathbb R)$. Clearly, $i_G([\Theta]_{c,b})=[\Theta]_b$.

The bounded cohomology $H^*_b(\pi_1(M),\mathbb R)$ is obtained by the cohomology of the complex $C^*_b(\widetilde M,\mathbb R)_\mathrm{alt}^{\pi_1(M)}$. Since $C^*_b(\widetilde M,\mathbb R)_\mathrm{alt}= F^*_{\mathrm{alt},b}(\widetilde M,\mathbb R)$,  the induced map $i_M : H^k_b(\pi_1(M),\mathbb R) \rightarrow H^k_b(\widetilde M;\pi_1(M),\mathbb R)$ is the identity map. 
Let $\widehat D : \widehat{\widetilde M}\rightarrow \overline{\mathcal X}$ be a $\rho$-equivariant map which maps $\widetilde M$ to $\mathcal X$. Then consider the following commutative diagram.
$$ \xymatrixcolsep{4pc}\xymatrix{
H^n(\overline \X;G,\mathbb R) \ar[r]^-{\widehat D^*} & H^n(\widehat{\widetilde{M}};\pi_1(M),\mathbb R) \ar[rd]^-{\pi^*} \\
H^n_b(\overline{\X};G,\mathbb R) \ar[r]^-{\widehat D^*_b} \ar[d]^-{res_\X} \ar[u]_-{\bar c} & 
H^n_b(\widehat{\widetilde{M}};\pi_1(M),\mathbb R) \ar[d]^-{res_M} \ar[rd]^-{\pi^*_b} \ar[u]_-{\hat c} & H^n(\overline M,\partial \overline  M,\mathbb R)\\
H^n_b(\X;G,\mathbb R) \ar[r]^-{D_b^*} &
H^n_b(\widetilde M;\pi_1(M),\mathbb R) &
H^n_b(\overline M,\partial \overline M,\mathbb R) \ar[l]_-{i^*_b} \ar[u]_-{c} \\
H^n_{c,b}(G,\mathbb R) \ar[u]_-{i_G} \ar[r]^-{\rho^*_b} & H^n_b(\pi_1(M),\mathbb R) \ar[u]_-{i_M}
}$$
where $\pi : \widetilde{\overline M} \rightarrow \widehat{\widetilde M}$ is the collapsing map. Note that the map $\rho^*_b$ in the bottom of the diagram is actually induced from the restriction map $D: \widetilde M \rightarrow \X$. However it does not depend on the choice of equivariant map but only on the homomorphism $\rho$. In other words, any continuous equivariant map from $\widetilde M$ to $\X$ gives rise to the same map $\rho^*_b: H^*_{c,b}(G,\mathbb R) \rightarrow H^*_b(\pi_1(M),\mathbb R)$. For this reason, we denote it by $\rho^*_b$ instead of $D^*_{c,b}$.

Note that $\pi$ induces a map $\pi^* : F^*_\mathrm{alt}(\widehat{\widetilde M},\mathbb R) \rightarrow F^*_\mathrm{alt}(\widetilde{\overline M},\mathbb R)$. It follows from the alternating property that the image of $\pi^*$ is contained in $C^*(\overline M,\partial \overline M,\mathbb R)$. Hence the map $\pi^* : H^n(\widehat{\widetilde M};\pi_1(M),\mathbb R) \rightarrow H^n(\overline M,\partial \overline M,\mathbb R)$ makes sense. One can understand $\pi^*_b : H^n_b(\widehat{\widetilde M};\pi_1(M),\mathbb R) \rightarrow H^n_b(\overline M,\partial \overline M,\mathbb R)$ in a similar way.

Noting that $\bar c([\overline \Theta]_b)=[\overline \Theta]$ and  $res_\X([\overline \Theta]_b)=[\Theta]_b$, it follows from the above commutative diagram that
{\setlength\arraycolsep{2pt}
\begin{eqnarray*}
((i^*_b)^{-1}\circ i_M \circ \rho_b^*)[\Theta]_{c,b} &=& ((i^*_b)^{-1}\circ D^*_b \circ i_G) [\Theta]_{c,b} \\
&=& ((i^*_b)^{-1}\circ D^*_b \circ res_\X) [\overline \Theta]_b\\
&=& ((i^*_b)^{-1}\circ res_M \circ \widehat D^*_b) [\overline \Theta]_b \\
&=& (\pi_b^* \circ \widehat D^*_b)[\overline \Theta]_b
\end{eqnarray*}}
Hence
{\setlength\arraycolsep{2pt}
\begin{eqnarray*}
\mathrm{Vol}_2(\rho)&=& \langle (c \circ (i^*_b)^{-1}\circ i_M \circ \rho_b^*)[\Theta]_{c,b}, [\overline M, \partial \overline M] \rangle \\
&=& \langle (c \circ \pi_b^* \circ \widehat D^*_b)[\overline \Theta]_b, [\overline M, \partial \overline M] \rangle \\
&=& \langle (\pi^* \circ \widehat D^* \circ \bar c)[\overline \Theta]_b, [\overline M, \partial \overline M] \rangle \\
&=& \langle (\pi^* \circ \widehat D^*)[\overline \Theta], [\overline M, \partial \overline M] \rangle \\
&=& \langle \widehat D^*[\overline \Theta], \pi_* [\overline M, \partial \overline M] \rangle \\
&=& \langle \widehat D^*[\overline \Theta], [\widehat M] \rangle \\
&=& \vol_4(\rho,D)
\end{eqnarray*}}
This completes the proof.
\end{proof}

Proposition \ref{prop:1} implies that the value $\vol_4(\rho,D)$ does not depend on the choice of continuous equivariant map. Hence from now on we use the notation $\vol_4(\rho):=\vol(\rho,D)$. Furthermore,
Proposition \ref{prop:1} allows us to interpret the invariant $\vol_2(\rho)$ in terms of a pseudo-developing map via $\vol_4(\rho)$ in the next section. Note that a pseudo-developing map for $\rho$ is a specific kind of $\rho$-equivariant continuous map $\widehat{\widetilde{M}} \rightarrow \overline \X$. 

\section{Pseudo-developing map and Definition {\bf D1}}\label{sec:pseudo}

Dunfield \cite{Du99} introduced the notion of pseudo-developing map in order to define the volume of representations $\rho : \pi_1(M)\rightarrow \mathrm{SO}(3,1)$ for a noncompact complete hyperbolic $3$-manifold $M$ of finite volume. We start by recalling the definition of pseudo-developing map.

\begin{definition}[Cone map]
Let $\mathcal A$ be a set, $t_0 \in \mathbb R$, and $Cone(\mathcal A)$ be the cone obtained from $\mathcal A\times [t_0,\infty]$ by collapsing $\mathcal A \times \{\infty\}$ to a point, called $\infty$. A map $\widehat D:Cone(\mathcal A) \rightarrow \overline \X$ is a \emph{cone map} if $\widehat D (Cone(\mathcal A))\cap \partial \X =\{\widehat D(\infty)\}$ and, for all $a \in \mathcal A$ the map $\widehat D|_{a\times [t_0,\infty]}$ is either the constant to $\widehat D(\infty)$ or the geodesic ray from $\widehat D(a,t_0)$ to $\widehat D(\infty)$, parametrized in such a way that the parameter $(t-t_0)$, $t\in [t_0,\infty]$, is the arc length.
\end{definition}

For each ideal point $v$ of $M$, fix a product structure $T_v \times [0,\infty)$ on the end relative to $v$. The fixed product structure induces a cone structure on a neighborhood of $v$ in $\widehat M$, which is obtained from $T_v \times [0,\infty]$ by collapsing $T_v \times \{\infty\}$ to a point $v$. We lift such structures to the universal cover. 
Let $\tilde v$ be an ideal point of $\widetilde M$ that projects to the ideal point $v$. Denote by $E_{\tilde v}$ the cone at $\tilde v$ that is homeomorphic to $P_{\tilde v} \times [0,\infty]$, where $P_{\tilde v}$ covers $T_v$ and $P_{\tilde v} \times \{\infty\}$ is collapsed to $\tilde v$.


\begin{definition}[Pseudo-developing map]\label{def:3.2}
Let $\rho : \pi_1(M) \rightarrow G$ be a representation. A \emph{pseudo-developing map} for $\rho$ is a piecewise smooth $\rho$-equivariant map $D : \widetilde M \rightarrow \X$. Moreover $D$ is required to extend to a continuous map
$\widehat D: \widehat{\widetilde{M}} \rightarrow \overline \X$ with the property that there exists $t \in \mathbb R^+$ such that for each end $E_{\tilde v}=P_{\tilde v} \times [0,\infty]$ of $\widehat{\widetilde{M}}$, the restriction of $\widehat D$ to $P_{\tilde v} \times [t,\infty]$ is a cone map. 
\end{definition}

\begin{definition}
A \emph{triangulation} of $\widehat M$ is an identification of $\widehat M$ with a complex obtained by gluing together with simplicial attaching maps. It is not required for the complex to be simplicial, but it is required that open simplicies embed. 
\end{definition}

Note that a triangulation of $\widehat M$ always exists and it lifts uniquely to a triangulation of $\widehat{\widetilde M}$.
Given a triangulation of $\widehat M$, one can define the straightening of pseudo-developing maps as follows.


\begin{definition}[Straightening map]
Let $\widehat M$ be triangulated.
Let $\rho:\pi_1(M)\rightarrow G$ be a representation and $D :\widetilde{M} \rightarrow \X$ a pseudo-developing map for $\rho$.
A straightening of $D$ is a continuous piecewise smooth $\rho$-equivariant map $Str( D):\widehat{\widetilde{M}} \rightarrow  \overline \X$ such that
\begin{itemize}
\item for each simplex $\sigma$ of the triangulation, $Str( D)$ maps $\widetilde \sigma$ to $Str( D \circ \widetilde \sigma)$,
\item for each end $E_{\tilde v}=P_{\tilde v}\times [0,\infty]$ there exists $t\in \mathbb R$ such that $Str( D)$ restricted to $P_{\tilde v} \times [t,\infty]$ is a cone map.
\end{itemize}
where $\widetilde \sigma$ is a lift of $\sigma$ to $\widehat{\widetilde M}$ and $Str(D \circ \widetilde \sigma)$ is the geodesic straightening of $ D\circ \widetilde \sigma : \Delta^n \rightarrow \overline \X$.
\end{definition}

Note that any straightening of a pseudo-developing map is also a pseudo-developing map.

\begin{lemma}
Let $\widehat M$ be triangulated. Let $\rho :\pi_1(M) \rightarrow G$ be a representation and $ D:\widetilde{M} \rightarrow \X$ a pseudo-developing map for $\rho$. Then a straightening $Str(D)$ of $D$ exists and furthermore, $Str(D) : \widehat{\widetilde{M}} \rightarrow \overline \X$ is always equivariantly homotopic to $\widehat D$ via a homotopy that fixes the vertices of the triangulation.
\end{lemma}

\begin{proof}
First, set $Str(D)(V)=f(V)$ for every vertex $V$ of the triangulation. Then extend $Str(D)$ to a map which is piecewise straight with respect to the triangulation. This is always possible because $\X$ is contractible.
Note that $\widehat D$ and $Str(D)$ agree on the ideal vertices of $\widehat{\widetilde{M}}$ and are equivariantly homotopic via the straight line homotopy between them. Hence it can be easily seen that the extension is a straightening of $D$.
\end{proof}

For any pseudo-developing map $D:{\widetilde{M}} \rightarrow \X$ for $\rho$, 
$$\int_M D^*\omega_\X$$ is always finite. This can be seen as follows. 
We stick to the notations used in Definition \ref{def:3.2}. We may assume that the restriction of $\widehat D$ to each $E_{\tilde v}=P_{\tilde v} \times [0,\infty]$ is a cone map. Choose a fundamental domain $F_0$ of $T_{v}$ in $P_{\tilde v}$. 
Then, there exists $t\in \mathbb R^+$ such that $$ \left|\int_{T_v \times [t,\infty)} D^*\omega_\X \right| =\vol_n (\mathrm{Cone}(D( F_0 \times \{t\}))) \leq \frac{1}{n-1}\vol_{n-1}(D(F_0\times \{t \}))$$ where $\vol_{n-1}$ denotes the $(n-1)$-dimensional volume.
The last inequality holds for any Hadamard manifold with sectional curvature at most $-1$. See \cite[Section 1.2]{Gro82}.
Hence the integral of $D^*\omega_\X$ over $M$ is finite.

\begin{definition}[{\bf D1}]
Let $D:{\widetilde{M}} \rightarrow \X$ be a pseudo-developing map for a representation $\rho : \pi_1(M) \rightarrow G$.
Define an invariant $\vol_1(\rho,D)$ by 
$$\vol_1(\rho,D)=\int_M D^*\omega_\X. $$
\end{definition}

In the case that $G=\mathrm{SO}(n,1)$, Francaviglia \cite{Fr04} showed that the definition {\bf D1} does not depend on the choice of pseudo-developing map. We give a self-contained proof for this in rank $1$ case.

\begin{proposition}\label{prop:14equi}
Let $\rho : \pi_1(M) \rightarrow G$ be a representation. Then for any pseudo-developing map $D :\widetilde M \rightarrow \X$,  $$\vol_1(\rho,D)=\vol_4(\rho).$$
Thus, $\vol_1(\rho,D)$ does not depend on the choice of pseudo-developing map.
\end{proposition}

\begin{proof}
Let $\mathcal T$ be a triangulation of $\widehat M$ with simplices $\sigma_1, \ldots, \sigma_N$. Then the triangulation gives rise to a fundamental cycle $\sum_{i=1}^N \sigma_i$ of $\widehat M$. Let $Str(D)$ be a straightening of $D$ with respect to the triangulation $\mathcal T$. 
Since $Str(D)$ is a $\rho$-equivariant continuous map,
we have
\begin{eqnarray*}
\vol_4(\rho):=\vol_4(\rho,D)&=&\langle Str(D)^*[\overline \Theta], [\widehat M] \rangle 
= \langle \overline \Theta, \sum_{i=1}^N Str(\widehat D(\sigma_i))\rangle \\
&=& \sum_{i=1}^N \int_{Str(\widehat D(\sigma_i))} \omega_\X 
= \int_M Str(D)^*\omega_\X.
\end{eqnarray*}
Since both $Str(D)$ and $\widehat D$ are pseudo-developing maps for $\rho$ that agree on the ideal points of $\widehat{\widetilde{M}}$, it can be proved, using the same arguments as the proof of \cite[Lemma 2.5.1]{Du99}, that $$\int_M Str(D)^*\omega_\X =\int_M D^*\omega_\X=\vol_1(\rho,D)$$
Finally we obtain the desired equality.
\end{proof}

\begin{remark}
While {\bf D1} is defined with only pseudo-developing map, the definition {\bf D4} is defined with any equivariant map. This is one advantage of the definition {\bf D4}. By Proposition \ref{prop:14equi}, the notation $\vol_1(\rho):=\vol_1(\rho,D)$ makes sense.
\end{remark}

\section{Lipschitz simplicial volume and Definition {\bf D3}}\label{sec:lipschitz}

In this section, $M$ is assumed to be a Riemannian manifold with finite Lipschitz simplicial volume.
Gromov \cite[Section 4.4]{Gro82} introduced the Lipschitz simplicial volume of Riemannian manifolds. One can define the Lipschitz constant for each singular simplex in $M$ by giving the Euclidean metrics on the standard simplices. Then the Lipschitz constant of a locally finite chain $c$ of $M$ is defined as the supremum of the Lipschitz constants of all singular simplices occurring in $c$. The Lipschitz simplicial volume of $M$ is defined by the infimum of the $\ell^1$-norms of all locally finite fundamental cycles with finite Lipschitz constant. Let $[M]_\mathrm{Lip}^{\ell^1}$ be the set of all locally finite fundamental cycles of $M$ with finite $\ell^1$-seminorm and finite Lipschitz constant. If $[M]_\mathrm{Lip}^{\ell^1}=\emptyset$, the Lipschitz simplicial volume of $M$ is infinite.

In the case that $[M]_\mathrm{Lip}^{\ell^1} \neq \emptyset$, S. Kim and I. Kim \cite{KK14} give a new definition of volume of representations as follows:
Given a representation $\rho : \pi_1(M) \rightarrow G$, $\rho$ induces a canonical pullback map
$\rho^*_b : H^*_{c,b}(G,\mathbb{R}) \rightarrow H^*_b(\pi_1(M),\mathbb{R})\cong H^*_b(M,\mathbb R)$ in continuous bounded cohomology.
Hence for any continuous bounded volume class $\omega_b \in H^n_{c,b}(G,\mathbb R)$, we obtain a bounded cohomology class $\rho^*_b(\omega_b)\in H^n_b(M,\mathbb{R})$.  
Then, the bounded cohomology class $\rho^*_b(\omega_b)$ can be evaluated on $\ell^1$-homology classes in $H^{\ell^1}_n(M,\mathbb{R})$ by the Kronecker products
$$ \langle\cdot ,\cdot \rangle : H^*_b(M,\mathbb{R}) \otimes H^{\ell^1}_*(M,\mathbb{R}) \rightarrow \mathbb{R}.$$
For more detail about this, see \cite{KK14}.

\begin{definition}[{\bf D3}] 
We define an invariant $\mathrm{Vol}_3(\rho)$ of $\rho$ by
$$\mathrm{Vol}_3(\rho) = \inf \langle \rho^*_b(\omega_b), \alpha \rangle$$
where the infimum is taken over all $\alpha\in [M]^{\ell^1}_\mathrm{Lip}$ and all $\omega_b \in H^n_{c,b}(G,\mathbb R)$ with $c(\omega_b)=\omega_{\mathcal X}$. 
\end{definition}

One advantage of {\bf D3} is to not need the isomorphism $H^n_b(\overline M,\partial \overline M,\mathbb R) \rightarrow H^n_b(\overline M,\mathbb R)$.
When $M$ admits the isomorphism above, we will verify that the definition {\bf D3} is eventually equivalent to the other definitions of volume of representations. 

\begin{lemma}\label{lem:indep}
Suppose that $M$ is a noncompact, connected, orientable, apherical, tame Riemannian manifold with finite Lipschitz simplicial volume and each end of $M$ has amenable fundamental group.
Then for any $\alpha \in [M]^{\ell^1}_\mathrm{Lip}$ and any continuous bounded volume class $\omega_b$, $$\langle \rho_b^* (\omega_b), \alpha \rangle = \langle (c\circ (i^*_b)^{-1} \circ \rho^*_b)(\omega_b), [\overline M,\partial \overline M] \rangle$$
\end{lemma}
\begin{proof}
When $M$ is a $2$-dimensional manifold, the proof is given in \cite{KK14}. Actually the proof in general case is the same.
We here sketch the proof for the reader's convenience.
Let $K$ be a compact core of $M$. Note that $K$ is a compact submanifold with boundary that is a deformation retract of $M$. Consider the following commutative diagram,
$$ \xymatrixcolsep{2pc}\xymatrix{
C^*_b(M,\mathbb{R}) &
C^*_b(\overline M,\mathbb{R}) \ar[l]_-{j_b^*} &
C^*_b(\overline M,\partial \overline M,\mathbb{R}) \ar[l]_-{i^*_b} \\
 & C^*_b(\overline M, \overline M-K, \mathbb{R})  \ar[u]^-{l_b^*} \ar[ru]_-{q_b^*} & }$$
where every map in the above diagram is the map induced from the canonical inclusion.
Every map in the diagram induces an isomorphism in bounded cohomology in $*\geq2$.
Thus, there exists a cocycle $z_b \in C^n_b(\overline M,\overline M-K,\mathbb{R})$ such that
$l^*_b([z_b]) = \rho^*_b(\omega_b)$.

Let $c=\sum_{i=1}^\infty a_i \sigma_i$ be a locally finite fundamental $\ell^1$-cycle with finite Lipschitz constant representing $\alpha \in [M]^{\ell^1}_\mathrm{Lip}$. Then, we have
$$\langle \rho^*_b(\omega_b),\alpha \rangle = \langle l_b^* ([z_b]), \alpha \rangle = \langle z_b, c \rangle.$$
Since $z_b$ vanishes on simplices with image contained in $\overline M-K$, we have $\langle z_b, c \rangle =\langle z_b, c|_K \rangle$
where  $c|_K=\sum_{\mathrm{im}\sigma_i \cap K \neq \emptyset} a_i \sigma_i$. It is a standard fact that $c|_K$ represents the relative fundamental class $[\overline M,\overline M-K]$ in $H_n(\overline M, \overline M-K,\mathbb{R})$ (see \cite[Theorem 5.3]{Loh07}.)
On the other hand, we have 
\begin{eqnarray*} 
\langle (c \circ (i^*_b)^{-1} \circ \rho_b^*)(\omega_b), [\overline M,\partial \overline M] \rangle &=& \langle (c\circ q^*_b)([z_b]), [\overline M,\partial \overline M]) \rangle \\
&=& \langle [z_b], q_*[\overline M,\partial \overline M] \rangle \\
&=& \langle [z_b], [\overline M,\overline M-K] \rangle=\langle z_b,c|_K \rangle .
\end{eqnarray*}
Therefore, we finally get the desired identity.
\end{proof}

By Lemma \ref{lem:indep} we can reformulate the definition {\bf D3} as follows.
$$ \vol_3(\rho)= \inf_{\omega_b} \langle (c\circ (i^*_b)^{-1} \circ \rho_b^*)(\omega_b), [\overline M,\partial \overline M] \rangle$$
where infimum is taken over all continuous bounded volume classes. Noting that $[\Theta]_{c,b} \in H^n_{c,b}(G,\mathbb R)$ is a continuous bounded volume class, it is clear that $$\vol_3(\rho)\leq \vol_2(\rho).$$

It is conjecturally true that the comparison map $H^n_{c,b}(G,\mathbb R) \rightarrow H^n_c(G,\mathbb R)$ is an isomorphism for any connected semisimple Lie group $G$ with finite center. Hence conjecturally, $\vol_2(\rho)=\vol_3(\rho)$. In spite of the absence of the proof of the conjecture, we will give a proof for $\vol_2(\rho)=\vol_3(\rho)$ by using the definition {\bf D4}.

\begin{lemma}\label{lem:extend}
Let $\omega_b \in H^n_{c,b}(G,\mathbb R)$ be a continuous bounded volume class. Let $f_b :\mathcal X^{n+1} \rightarrow \mathbb R$ be a continuous bounded alternating $G$-invariant cocycle representing $\omega_b$. Then $f_b$ is extended to a bounded alternating $G$-invariant cocycle $\bar f_b : \overline{\mathcal X}^{n+1} \rightarrow \mathbb R$. Furthermore, $\bar f_b$ is uniformly continuous on $\X^n \times \{\xi\}$ for any $\xi \in \partial \X$.
\end{lemma}
\begin{proof}
For any $(\bar x_0, \ldots, \bar x_n) \in \overline{\mathcal X}^{n+1}$, define
$$\bar f_b(\bar x_0,\ldots, \bar x_n) = \lim_{t\rightarrow \infty} f_b(c_0(t),\ldots,c_n(t)),$$ 
where each $c_i(t)$ is a geodesic ray toward $\bar x_i$. Here, for $x \in \X$, we say that $c : [0,\infty) \rightarrow \X$ is a geodesic ray toward $x$ if there exists $t\in [0,\infty)$ such that the restriction map $c|_{[0,t]}$ of $c$ to $[0,t]$ is a geodesic with $c(t)=x$ and $c|_{[t,\infty)}$ is constant to $x$.
Then it is clear that $\bar f_b(x_0,\ldots,x_n)=f_b(x_0,\ldots, x_n)$ for $(x_0,\ldots,x_n) \in \mathcal X^{n+1}$.
To see the well-definedness of $\bar f_b$, we need to show that for other geodesic rays $c_i'(t)$ toward $\bar x_i$, 
\begin{eqnarray}\label{eqn:welldefine} \lim_{t\rightarrow \infty} f_b(c_0(t),\ldots,c_n(t))=\lim_{t\rightarrow \infty} f_b(c_0'(t),\ldots,c_n'(t)).\end{eqnarray}
Note that the limit always exists because $f_b$ is bounded.
In rank $1$ case, the distance between two geodesic rays with the same endpoint decays exponentially to $0$ as they go to the endpoint. Moreover since $f_b$ is $G$-invariant and $G$ transitively acts on $\X$, $f_b$ is uniformly continuous on $\X^{n+1}$. Thus, for any $\epsilon>0$ there exists some number $T>0$ such that 
$$ | f_b(c_0(t),\ldots,c_n(t)) - f_b(c_0'(t),\ldots,c_n'(t)) | <\epsilon$$
for all $t>T$. This implies (\ref{eqn:welldefine}) and hence $\bar f_b$ is well-defined.

The alternating property of $\bar f_b$ actually comes from $f_b$. Due to the alternating property of $f_b$, we have \begin{eqnarray*} 
\bar f_b(\bar x_0, \ldots, \bar x_i,\ldots, \bar x_j,\ldots, \bar x_n) &=& \lim_{t\rightarrow \infty} f_b(c_0(t),\ldots,c_i(t),\ldots,c_j(t),\ldots, c_n(t)) \\
&=&\lim_{t\rightarrow \infty} -f_b(c_0(t),\ldots,c_j(t),\ldots,c_i(t),\ldots, c_n(t))\\
&=&-\bar f_b(\bar x_0,\ldots, \bar x_j,\ldots, \bar x_i,\ldots, \bar x_n)
\end{eqnarray*}
Therefore we conclude that $\bar f_b$ is alternating. 
The boundedness and $G$-invariance of $\bar f_b$ immediately follows from the boundedness and $G$-invariance of $f_b$. Furthermore, it is easy to check that $\bar f_b$ is a cocycle by a direct computation.

Now it remains to prove that $\bar f_b$ is uniformly continuous on $\X^n\times \{\xi\}$. It is obvious that $\bar f_b$ is continuous on $\X^n\times \{\xi\}$. Noting that the parabolic subgroup of $G$ stabilizing $\xi$ acts on $\X$ transitively, it can be easily seen that $\bar f_b$ is uniformly continuous on $\X^n\times \{\xi\}$.
\end{proof}

The existence of $\bar f_b$ allows us to reformulate $\vol_3$ in terms of $\vol_4$.
Following the proof of Proposition \ref{prop:1}, we get 
\begin{eqnarray}\label{eqn:A} \langle (c\circ (i^*_b)^{-1} \circ \rho_b^*)(\omega_b), [\overline M,\partial \overline M] \rangle = \langle \widehat D^* [\bar f_b], [\widehat M] \rangle  \end{eqnarray} 
The last term $\langle \widehat D^* [\bar f_b], [\widehat M] \rangle$ above is computed by $\langle \widehat D^*\bar f_b, \widehat c \rangle$ for any equivariant map $\widehat D$ and fundamental cycle $\widehat c$ of $\widehat M$. By choosing proper equivariant map and fundamental cycle, we will show that $\langle \widehat D^* [\bar f_b], [\widehat M] \rangle$ does not depend on the choice of continuous bounded volume class.

\begin{proposition}\label{lem:indepwb}
Let $\omega_b$ and $\omega_b'$ be continuous bounded volume classes. Let $\bar f_b$ and $\bar f_b'$ be the bounded alternating cocycles in $F^n_\mathrm{alt}(\overline X;G,\mathbb R)$ associated with $\omega_b$ and $\omega_b'$ respectively as in Lemma \ref{lem:extend}. Then 
$$\langle \widehat D^* [\bar f_b], [\widehat M] \rangle=\langle \widehat D^* [\bar f_b'], [\widehat M] \rangle.$$
\end{proposition}
\begin{proof}
It suffices to prove that for some $\rho$-equivariant map $\widehat D :\widehat{\widetilde M} \rightarrow \overline \X$ and fundamental cycle $\widehat c$ of $\widehat M$,
$$\langle \widehat D^*\bar f_b, \widehat c\rangle = \langle \widehat D^*\bar f_b', \widehat c\rangle.$$
To show this, we will prove that for some sequence $(\widehat c_k)_{k\in \mathbb N}$ of fundamental cycles of $\widehat M$  $$\lim_{k\rightarrow \infty} \left( \langle \widehat D^*\bar f_b, \widehat c_k \rangle - \langle \widehat D^*\bar f_b', \widehat c_k\rangle \right)=0.$$

Let $v_1,\ldots,v_s$ be the ideal points of $M$. 
As in Section \ref{sec:pseudo}, fix a product structure $T_{v_i} \times [0,\infty]$ on the end relative to $v_i$ for each $i=1,\ldots,s$ and then lift such structures to the universal cover. We stick to the notations used in Section \ref{sec:pseudo}.
Set $$M_k = M-\cup_{i=1}^s T_{v_i} \times (k,\infty].$$
Then $(M_k)_{k\in \mathbb N}$ is an exhausting sequence of compact cores of $M$. The boundary $\partial M_k$ of $M_k$ consists of $\cup_{i=1}^s T_{v_i} \times \{k\}$. Let $\mathcal T_0$ be a triangulation of $M_0$. Then we extend it to a triangulation on $\widehat M$ as follows. First note that $\mathcal T_0$ induces a triangulation on each $T_{v_i}$. 
Let $\tau$ be an $(n-1)$-simplex of the induced triangulation on $T_{v_i}$ for some $i \in \{1,\ldots,s\}$.
Then we attach $\pi(\tau \times [0,\infty])$ to $T_{v_i}\times\{0\}$ along $\tau \times \{0\}$ where $\pi :\overline M \rightarrow \widehat M$ be the collapsing map.
Since $\pi$ is an embedding on $\tau \times [0,\infty)$ and $\pi$ maps $\tau \times \{\infty\}$ to the ideal point $v_i$,  
it can be easily seen that $cone(\tau):=\pi(\tau \times [0,\infty])$ is an $n$-simplex. Hence we can obtain a triangulation of $\widehat M$ by attaching each $cone(\tau)$ to $\partial M_0$, which is denoted by $\widehat{\mathcal T_0}$.

Next, we extend $\mathcal T_0$ to a triangulation of $M_k$. In fact, $M_k$ is decomposed as follows. $$M_k = M_0 \cup \bigcup_{i=1}^s T_{v_i} \times [0,k].$$ Hence attach each $\tau \times [0,k]$ to $M_0$ along $\tau \times\{0\}$ and then triangulate $\tau \times [0,k]$ by using the prism operator \cite[Chapter 2.1]{Hatcher}. Via this process, we obtain a triangulation of $M_k$, denoted by $\mathcal T_k$. Note that $\mathcal T_0$ and $\mathcal T_k$ induce the same triangulation on each $T_{v_i}$. In addition, one can obtain a triangulation $\widehat{\mathcal T_k}$ of $\widehat M$ from $\mathcal T_k$ in a similar way that $\widehat{\mathcal T_0}$ is obtained from $\mathcal T_0$ as above.

Let $c_k$ be the relative fundamental class of $(M_k,\partial M_k)$ induced from $\mathcal T_k$. Then it can be seen that 
$$\widehat c_k = c_k + (-1)^{n+1}cone(\partial c_k)$$
is the fundamental cycle of $\widehat M$ induced from $\widehat{\mathcal T_k}$. Any simplex occurring in $c_k$ is contained in $M_k$.
Now we choose a pseudo-developing map $\widehat D :\widehat{\widetilde M} \rightarrow \overline X$. 
Let $\tilde v_i$ be a lift of $v_i$ to $\widehat{\widetilde M}$. Let $P_{\tilde v_i} \times [0,\infty]$ be the cone structure of a neighborhood of $\tilde v_i$ where $P_{\tilde v_i} $ covers $T_{v_i}$ and $P_{\tilde v_i}\times \{\infty\}$ is just the ideal point $\tilde v_i$. We may assume that $\widehat D$ is a cone map on each $P_{\tilde v_i} \times [0,\infty]$. Let $\tilde c_k$ be a lift of $c_k$ to a cochain in $\widetilde M$ and $\widetilde{\partial c_k}$ be a lift of $\partial c_k$. Let $\tau \times \{0\}$ be an $(n-1)$-simplex in $T_{v_i} \times \{0\}$ occurring in $\partial c_0$ and 
$\tilde \tau$ be a lift of $\tau$ to $P_{\tilde v_i}$. Then $\tilde \tau \times \{k\}$ is a lift of $\tau \times \{k\} \in \partial c_k$. Since $\widehat D$ is a cone map on $P_{\tilde v_i} \times [0,\infty]$, $D(\tilde \tau \times [0,\infty])$ is the geodesic cone over $\tilde \tau \times \{0\}$ with top point $\tilde v_i$ in $\overline \X$. Hence the diameter of $D(\tilde \tau \times \{k\})$ decays exponentially to $0$ as $k \rightarrow \infty$ for each $\tau$. 

By a direct computation, we have
\begin{eqnarray*}
\langle \widehat D^*\bar f_b - \widehat D^*\bar f_b' , \widehat c_k\rangle &=& \langle \widehat D^*\bar f_b - \widehat D^*\bar f_b', \tilde c_k \rangle +  (-1)^{n+1}\langle \widehat D^*\bar f_b - \widehat D^*\bar f_b', cone(\widetilde{\partial c_k}) \rangle \\
&=& \langle \bar f_b - \bar f_b', \widehat D_*(\tilde c_k) \rangle +  (-1)^{n+1}\langle \bar f_b - \bar f_b', \widehat D_*(cone(\widetilde{\partial c_k}))\rangle \\
&=& \langle f_b - f_b', D_*(\tilde c_k) \rangle +  (-1)^{n+1}\langle \bar f_b - \bar f_b', \widehat D_*(cone(\widetilde{\partial c_k}))\rangle 
\end{eqnarray*}
The last equality comes from the fact that $\widehat D_*(\tilde c_k)$ is a singular chain in $\X$. Since $f_b$ and $f_b'$ are continuous bounded alternating cocycles representing the continuous volume class $\omega_{\mathcal X} \in H^n_c(G,\mathbb R)$, there is a continuous alternating $G$-invariant function $\beta : \X^n \rightarrow \mathbb R$ such that $f_b -f_b' =\delta \beta$. Hence 
$$\langle f_b - f_b', D_*(\tilde c_k) \rangle = \langle \delta \beta, D_*(\tilde c_k) \rangle = \langle \beta, \partial D_*(\tilde c_k) \rangle = \langle \beta,  D_*(\widetilde{\partial c_k}) \rangle.$$

As observed before, since the diameter of all simplices occurring in $D_*(\widetilde{\partial c_k})$ decays to $0$ as $k \rightarrow \infty$ and moreover $\beta$ is uniformly continuous on $\X$, we have $$\lim_{k\rightarrow \infty}  \langle \beta,  D_*(\widetilde{\partial c_k}) \rangle =0$$

Note that $D(cone(\tilde \tau \times \{k\}))$ is the geodesic cone over $D(\tilde \tau \times \{k\})$ with top point $\tilde v_i$. By Lemma \ref{lem:extend}, both $\bar f_b$ and $\bar f_b'$ are uniformly continuous on $\X^n \times \{\tilde v_i\}$.
Since the diameter of $D(\tilde \tau \times \{k\})$ decays to $0$ as $k\rightarrow \infty$,
$$\lim_{k \rightarrow \infty} \langle \bar f_b, D(cone(\tilde \tau \times \{k\})) \rangle =\lim_{k \rightarrow \infty} \langle \bar f_b', D(cone(\tilde \tau \times \{k\})) \rangle = 0.$$
Applying this to each $\tau$, we can conclude that
   $$\lim_{k \rightarrow \infty} \langle \bar f_b,D_*(cone(\widetilde{\partial c_k})) \rangle =\lim_{k \rightarrow \infty} \langle \bar f_b', D_*(cone(\widetilde{\partial c_k}))\rangle =0.$$
In the end, it follows that $$\lim_{k\rightarrow \infty} \langle \widehat D^*\bar f_b - \widehat D^*\bar f_b' , \widehat c_k \rangle=0.$$ 
As we mentioned, the value on the left hand side above does not depend on $\widehat c_k$. Thus we can conclude that $\langle \widehat D^*\bar f_b - \widehat D^*\bar f_b' , \widehat c_k \rangle=0$. 
This implies that $\langle \widehat D^*\bar f_b, \widehat c \rangle = \langle \widehat D^*\bar f_b', \widehat c\rangle$ for any fundamental cycle $\widehat c$ of $\widehat M$, which completes the proof.
\end{proof}

Combining Proposition \ref{lem:indepwb} with (\ref{eqn:A}), Proposition \ref{prop:indepwb} immediately follows.

\begin{proposition}
The definitions of {\bf D3} and {\bf D4} are equivalent.
\end{proposition}
\begin{proof}
By Lemma \ref{lem:indep} and Proposition \ref{prop:1}, we have 
\begin{eqnarray*}
\vol_3(\rho) &=& \inf \{ \langle \rho^*_b(\omega_b),\alpha \rangle \ | \ c(\omega_b)=\omega_{\X} \text{ and } \alpha\in [M]_\mathrm{Lip}^{\ell^1} \} \\
&=& \inf \{ \langle (c\circ (i^*_b)^{-1} \circ \rho^*_b)(\omega_b), [\overline M,\partial \overline M] \rangle \ | \ c(\omega_b)=\omega_\X \} \\
&=& \inf \{  \langle \widehat D^* [\bar f_b], [\widehat M] \rangle  \ | \ c(\omega_b)=\omega_\X \} \\
&=&  \langle \widehat D^* [\overline \Theta], [\widehat M] \rangle  \\
&=& \vol_4(\rho),
\end{eqnarray*}
which completes the proof.
\end{proof}

\end{document}